\documentclass[12pt,a4paper,draft]{article}
\providecommand{\abs}[1]{\left| #1 \right| }
\providecommand{\norm}[1]{\lVert#1\rVert}
\newcommand{\lt}{\leq}
\newcommand{\gt}{\geq}
\newcommand{\R}{\mathbb{R}}
\newcommand{\N}{\mathbb{N}}

\usepackage[leqno]{amsmath}
\usepackage{amssymb,amsthm,upref,amscd}
\usepackage[T1]{fontenc}
\usepackage{times}
\usepackage{cite}%ergaenzt von T.Weth
\usepackage{amsfonts}
\newtheorem{theorem}{Theorem}[section]

\newtheorem{lemma}[theorem]{Lemma}

\numberwithin{equation}{section}

\newcommand{\tJ}{\tilde{J}}
\DeclareMathOperator{\im}{Im}

\newenvironment{altproof}[1]
{\noindent%\addvspace{0.3cm}
{\em Proof of {#1}}.}
{\nopagebreak\mbox{}\hfill $\Box$\par\addvspace{0.5cm}}

\allowdisplaybreaks

\begin{document}
\title{Normalized solutions of nonlinear Schr\"odinger equations}
\author{Thomas Bartsch \and S\'ebastien de Valeriola}
\date{}
\maketitle

\begin{abstract}
We consider the problem
\begin{equation*}
\begin{cases}
-\Delta u - g(u) = \lambda u, \\
u \in H^1(\R^N), \ \int_{\R^N} u^2 = 1,\ \lambda\in\R,
\end{cases}
\end{equation*}
in dimension $N\ge2$. Here $g$ is a superlinear, subcritical, possibly nonhomogeneous, odd nonlinearity. We deal with the case where the associated functional is not bounded below on the $L^2$-unit sphere, and we show the existence of infinitely many solutions.
\end{abstract}



\section{Introduction}
In this note we consider the nonlinear eigenvalue problem
\begin{equation}\label{problem}
\begin{cases}
-\Delta u - g(u) = \lambda u, \\
u \in H^1(\R^N), \ \int_{\R^N} u^2 = 1,\ \lambda\in\R,
\end{cases}
\end{equation}
in dimension $N\ge2$. The nonlinearity $g:\R\to\R$ is superlinear, subcritical, and possibly nonhomogeneous. A model nonlinearity is
\begin{equation}\label{model}
g(u) = \left(\sum_{i=1}^k |u|^{p_i-2}\right)u,\quad 2<p_1<\ldots<p_k<2^*,
\end{equation}
where $2^* = 2N/(N-2)$ if $N\gt 3$ and $\infty$ if $N=2$, the critical Sobolev exponent.

This problem possesses many physical motivations, e.~g.\ it appears in models for Bose-Einstein condensation (see \cite{Kevrekidis-etal:2008}). Looking for standing wave solutions $\Psi(t,x)=e^{imt}u(x)$ of the dimensionless nonlinear Schr\"odinger equation
$$
i\Psi_t-\Delta_x\Psi = f(|\Psi|)\Psi
$$
one is lead to problem \eqref{problem} with $g(u)=f(|u|)u$. As in these physical frameworks $\Psi$ is a wave function, it seems natural to search for \emph{normalized} solutions, i.~e.\ solutions of the equation satisfying $\int_{\R^N} u^2 = 1$.

If $g$ is homogeneous ($k=1$ in \eqref{model}) then one can use the classical results from \cite{Berestycki-Lions:1983a,Berestycki-Lions:1983b}, for instance, to solve $-\Delta u + u = g(u)$, and then rescale $u$ in order to obtain normalized solutions of \eqref{problem}. This does not work for a general nonlinearity, it fails already in the case $k\ge2$ in \eqref{model}. If $g$ is not homogeneous and does not grow too fast (for $g$ as in \eqref{model} this means all $p_i<2+\frac4N$) then one can minimize the associated functional
\begin{equation}\label{def:J}
J(u)=\frac12\int_{\R^N} |\nabla u|^2 - \int_{\R^N} G(u),\quad\text{with } G(t)=\int_0^t g(s)\,ds,
\end{equation}
on the $L^2$-unit sphere $S=\{u\in H^1_{\text{rad}}(\R^N):\int_{\R^N} u^2 = 1\}$ to obtain a solution. Here $H^1_{\text{rad}}(\R^N)$ denotes the space of radial $H^1$-functions. The parameter $\lambda$ appears as Lagrange multiplier. Rather general conditions on $g$ which allow minimization, even in a nonradial setting, can be found in \cite{Hajaiej-Stuart:2005} and the references therein. If $g$ is odd, as in the case $g(u)=f(|u|)u$ appearing in applications, and if $g$ does not grow too fast then one can obtain infinitely many solutions using classical min-max arguments based on the Krasnoselski genus.

However for fast growing $g$, $J$ is not bounded below on $S$, hence minimization doesn't work. Moreover, the genus of the sublevel sets $J^c=\{ u\in S: J(u)\le c\}$ is always infinite, so the Krasnoselski genus arguments do not apply. In \cite{Je97}, Jeanjean was able to treat nonhomogeneous, fast growing nonlinearities and showed the existence of  one solution of \eqref{problem} using a mountain pass structure for $J$ on $S$. The object of this short note is to prove that for the same class of nonlinearities considered in \cite{Je97}, \eqref{problem} actually has infinitely many solutions.

In order to state our result we recall the assumptions on the function $g$ made in \cite{Je97}:
\begin{itemize}
\item[($H_1$)] $g : \R \to \R$ is continuous and odd,
\item[($H_2$)] there exists $\alpha,\beta \in \R$ satisfying
\begin{equation*}
2+\frac{4}{N} < \alpha \lt \beta < 2^*
\end{equation*}
such that
\begin{equation*}
0 < \alpha G(s) \lt g(s) s \lt \beta G(s).
\end{equation*}
\end{itemize}

The condition $G>0$ in ($H_2$) is not stated in \cite{Je97} but used implicitely.

\begin{theorem}\label{main}
If assumptions ($H_1$) and ($H_2$) hold, then problem \eqref{problem} possesses an unbounded sequence of pairs of radial solutions $(\lambda_n, \pm u_n)$.
\end{theorem}

The proof is based on variational methods applied to the functional $J$ constrained to $S$. We shall present a new linking geometry for constrained functionals which is motivated by the fountain theorem \cite[Theorem~2.5]{Ba93b}; see also \cite[Section~3]{Willem:1996}. The classical symmetric mountain pass theorem applies to functionals on Banach spaces, not on spheres. Another difficulty due to the constraint is that $J|_S$ does not satisfy the Palais-Smale condition although the embedding $H^1_{\text{rad}}(\R^N) \hookrightarrow L^p(\R^N)$ of the space of radial $H^1$-functions into the $L^p$-spaces is compact for $2<p<2^*$. In fact, there exist bounded Palais-Smale sequences for $J|_S$ converging weakly to 0, and there may exist unbounded Palais-Smale sequences.

\section{Proof of Theorem \ref{main}}

In order to recover some compacity, we will work in $E = H^1_{\text{rad}}(\R^N)$, provided with the standard scalar product and norm: $\|u\|^2=|\nabla u|_2^2+|u|_2^2$. Here and in the sequel we write $|u|_p$ to denote the $L^p$-norm. As we look for normalized solutions, we consider the functional $J$ constrained to the $L^2$-unit sphere in $E$:
\begin{equation*}
J_S : S = \{ u \in E : |u|_2 = 1\} \to \R ,\quad
u \mapsto \frac12 \int_{\R^N} \abs{\nabla u}^2 - \int_{\R^N} G(u).
\end{equation*}
Observe that $\nabla J_S(u)=\nabla J(u)-\lambda_u u$ for some $\lambda_u \in \R$.

The main theorem's proof will follow from several lemmas. We fix a strictly increasing sequence of finite-dimensional linear subspaces $V_n\subset E$ such that $\bigcup_n V_n$ is dense in $E$.

\begin{lemma}
For $2 < p < 2^*$ there holds:
\begin{equation*}
\mu_n(p)
 = \inf_{u \in V_{n-1}^\perp} \frac{\int_{\R^N} (\abs{\nabla u}^2 + u^2)}{\left(\int_{\R^N} \abs{u}^p\right)^{2/p}}
 = \inf_{u \in V_{n-1}^\perp} \frac{\|u\|^2}{|u|_p^2}
 \to \infty\quad\text{as $n \to \infty$.}
\end{equation*}
\end{lemma}

\begin{proof}
Arguing by contradiction, suppose there exists a sequence $(u_n) \subset E$ such that $u_n \in V_{n-1}^{\perp}$, $|u_n|_p = 1$ and $\norm{u_n} \to c < \infty$. Then there exists $u \in E$ with $u_n \rightharpoonup u$ in $E$ and $u_n \to u$ in $L^p$ up to a subsequence. Let $v \in E$ and $(v_n) \subset E$ such that $v_n \in V_{n-1}$ and $v_n \to v$ in $V$. We have, in $E$,
\begin{equation*}
\abs{\langle u_n,v\rangle } \lt \abs{\langle u_n, v - v_n \rangle } + \abs{\langle u_n, v_n \rangle } \lt \norm{u_n} \norm{v-v_n} \to 0
\end{equation*}
so that $u_n \rightharpoonup 0 = u $, while $|u|_p = 1$, a contradiction.
\end{proof}

We introduce now the constant
\begin{equation*}
K = \max_{x >0} \frac{\abs{G(x)}}{\abs{x}^\alpha + \abs{x}^\beta},
\end{equation*}
which is well defined thanks to assumption ($H_2$). For $n\in\N$ we define
\begin{equation*}
\rho_n = \frac{M_n^{\beta/(2(\beta-2))}}{L^{1/(\beta-2)}},
\end{equation*}
where
\begin{equation*}
M_n = \left( \mu_n(\alpha)^{-\alpha/2} + \mu_n(\beta)^{-\beta/2} \right)^{-2/\beta} \quad \text{and} \quad
L = 3 K \max_{x>0} \frac{(1+x^2)^{\beta/2}}{1+x^\beta}.
\end{equation*}
We also define
\begin{equation*}
B_n = \left\{ u \in V_{n-1}^{\perp} \cap S : |\nabla u|_2 = \rho_n \right\}.
\end{equation*}
Then we have:

\begin{lemma}
$\displaystyle \inf_{u \in B_n} J(u) \to \infty$ as $n \to \infty$.
\end{lemma}

\begin{proof}
For any $u \in B_n$, we deduce, using the preceding lemma with $p = \alpha$ and $p = \beta$,
\begin{align*}
J(u) & = \frac 12 \int_{\R^N} \abs{\nabla u}^2 - \int_{\R^N} G(u)
 \gt \frac 12 \int_{\R^N} \abs{\nabla u}^2 - K \int_{\R^N} \abs{u}^\alpha - K \int_{\R^N} \abs{u}^\beta \\
& \gt \frac 12 \int_{\R^N} \abs{\nabla u}^2 - \frac{K}{\mu_n(\alpha)^{\alpha/2}} \left( \int_{\R^N} \abs{\nabla u}^2 + 1 \right)^{\alpha/2} \\
&\hspace{1.5cm} - \frac{K}{\mu_n(\beta)^{\beta/2}} \left( \int_{\R^N} \abs{\nabla u}^2 + 1 \right)^{\beta/2} \\
& \gt \frac 12 \int_{\R^N} \abs{\nabla u}^2 - \frac{K}{M_n^{\beta/2}} \left( \int_{\R^N} \abs{\nabla u}^2 + 1 \right)^{\beta/2}  \\
& \gt \frac 12 \int_{\R^N} \abs{\nabla u}^2 - \frac{L}{3 M_n^{\beta/2}} \left(\left( \int_{\R^N} \abs{\nabla u}^2\right)^{\beta/2} + 1 \right)  \\
& = \frac 12 \rho_n^2 - \frac{L}{3 M_n^{\beta/2}} \rho_n^{\beta} + o(1)
 = \left( \frac12 - \frac 13 \right) \rho^2_n + o(1) \to \infty.
\end{align*}
\end{proof}
Let $P_{n-1} : E \to V_{n-1}$ be the orthogonal projection, and set
\begin{equation*}
h_n : S \to V_{n-1} \times \R^+ ,\quad u \mapsto \left( P_{n-1} u,|\nabla u|_2\right).
\end{equation*}
Then clearly $B_n = h_n^{-1}(0,\rho_n)$. With $\pi:V_{n-1} \times \R^+ \to \R^+$ denoting the projection we define
\begin{multline*}
\Gamma_n = \Big\{ \gamma : [0,1]\times (S \cap V_n) \to S \mid \gamma \text{ is continuous, odd in $u$ and such that }\\
\forall u:\ \pi\circ h_n \circ\gamma(0,u) < \rho_n/2,\
\pi\circ h_n \circ\gamma(1,u) > 2\rho_n \Big\}.
\end{multline*}
It is easy to see that $\Gamma_n \neq \emptyset$. To describe a particular element $\gamma\in\Gamma_n$, let
$$
m:\R\times E \to E, \quad m(s,u)=s*u,
$$
be the action of the group $\R$ on $E$ defined by
\begin{equation*}
(s * u)(x) = e^{sN/2} u(e^s x) \quad \forall s\in \R,\ u\in E,\ x \in \R^N.
\end{equation*}
Observe that $s*u\in S$ if $u\in S$. The map $\gamma(t,u) = (2s_nt-s_n)*u$ lies in $\Gamma_n$ for $s_n>0$ large.

We now need the following linking property.

\begin{lemma}\label{cohomo}
For every $\gamma \in \Gamma_n$, there exists $(t,u) \in [0,1]\times (S \cap V_n)$ such that $\gamma(t,u) \in B_n$.
\end{lemma}

For the proof of this lemma we need to recall some properties of the cohomological index for spaces with an action of the group $G = \{-1,1\}$. This index goes back to \cite{CoFl62} and has been used in a variational setting in \cite{FaRa77}. It associates to a $G$-space $X$ an element $i(X) \in \N_0 \cup \{\infty\}$. We only need the following properties.
\begin{itemize}
\item[($I_1$)] If $G$ acts on $\mathbb{S}^{n-1}$ via multiplication then $i(\mathbb{S}^{n-1}) = n$.
\item[($I_2$)] If there exists an equivariant map $X \to Y$ then $i(X) \lt i(Y)$.
\item[($I_3$)] Let $X = X_0 \cup X_1$ be metrisable and $X_0, X_1 \subset X$ be closed $G$-invariant subspaces. Let $Y$ be a $G$-space and consider a continuous map $\phi : [0,1] \times Y \to X$ such that each $\phi_t = \phi(t,\cdot) : Y \to X$ is equivariant. If $\phi_0(Y) \subset X_0$ and $\phi_1(Y) \subset X_1$ then
\begin{equation*}
i(\im(\phi) \cap X_0 \cap X_1) \gt i(Y).
\end{equation*}
\end{itemize}
Properties ($I_1$) and ($I_2$) are standard and hold also for the Krasnoselskii genus. Property $(I_3)$ has been proven in \cite[Corollary 4.11, Remark 4.12]{Ba93a}. We can now prove Lemma~\ref{cohomo}.

\begin{proof}
We fix $\gamma \in \Gamma_n$, and consider the map
\begin{equation*}
\phi = h_n \circ \gamma : [0,1] \times (S \cap V_n) \to V_{n-1} \times \R^+ =: X.
\end{equation*}
Since
$$
\phi_0(S \cap V_n) \subset V_{n-1} \times (0,\rho_n] =: X_0
$$
and
$$
\phi_1(S \cap V_n) \subset V_{n-1} \times [\rho_n,\infty) =: X_1,
$$
it follows from ($I_1$) -- ($I_3$) that
\begin{equation*}
i(\im(\phi) \cap X_0 \cap X_1) \gt i(S \cap V_n) = \dim V_n.
\end{equation*}
If there would not exist $(t,u) \in [0,1] \times (S \cap V_n)$ with $\gamma(t,u) \in B_n$, then
\begin{equation*}
\im(\phi) \cap X_0 \cap X_1 \subset (V_{n-1} \setminus \{0\})\times  \{\rho_0\}.
\end{equation*}
Now ($I_1$), ($I_2$) imply that
\begin{equation*}
i(\im(\phi) \cap X_0 \cap X_1) \lt i( (V_{n-1} \setminus \{0\})\times  \{\rho_0\}) = \dim V_{n-1},
\end{equation*}
contradicting $\dim V_{n-1} < \dim V_n$.
\end{proof}

It follows from Lemma~\ref{cohomo} that
\begin{equation}\label{c_n}
c_n = \inf_{\gamma \in \Gamma_n} \max_{\substack{t \in [0,1] \\ u \in S \cap V_n}} J(\gamma(t,u))
    \gt \inf_{ u \in B_n} J(u) \to \infty.
\end{equation}
We will show that $c_n$ is a critical value of $J$, which finishes the proof of Theorem~\ref{main}. We fix $n$ from now on.

\begin{lemma}\label{exist-PS}
There exists a Palais-Smale sequence $(u_k)_k$ for $J_S$ at the level $c_n$ satisfying
\begin{equation}\label{pohozaev}
|\nabla u_k|_2^2 + N\int_{\R^N} G(u_k) -\frac{N}{2}\int_{\R^N} g(u_k)u_k \to 0.
\end{equation}
\end{lemma}

 For the proof we recall the stretched functional from \cite{Je97}:
\begin{equation*}
\tJ : \R \times E \to \R ,\quad (s,u) \mapsto J(s*u).
\end{equation*}
Now we define
\begin{multline*}
\tilde{\Gamma}_n =  \Big\{ \tilde{\gamma} : [0,1]\times (S \cap V_n) \to \R \times S \mid \tilde{\gamma} \text{ is continuous, odd in $u$,} \\
\text{and such that } m \circ \tilde{\gamma} \in \Gamma_n \Big\},
\end{multline*}
where $m(s,u) = s*u$, and
\begin{equation*}
\tilde{c}_n = \inf_{\tilde{\gamma} \in \tilde{\Gamma}_n} \max_{\substack{t \in [0,1] \\ u \in S \cap V_n}} \tJ(\tilde{\gamma}(t,u)).
\end{equation*}

\begin{lemma}
We have $\tilde{c}_n = c_n$.
\end{lemma}

\begin{proof}
The maps
\begin{equation*}
\Phi : \Gamma_n \to \tilde{\Gamma}_n ,\quad \gamma \mapsto [(0,\gamma):\ (t,u)\mapsto(0,\gamma(t,u))],
\end{equation*}
and
\begin{equation*}
\Psi : \tilde{\Gamma}_n \to \Gamma_n ,\quad \tilde\gamma \mapsto [m\circ\gamma:\ (t,u)\mapsto m(\tilde\gamma(t,u))],
%(\tilde{\gamma}_1,\tilde{\gamma}_2) \mapsto \tilde{\gamma}_1 * \tilde{\gamma}_2,
\end{equation*}
satisfy
\begin{equation*}
\tilde{J} (\Phi(\gamma)(t,u)) = J(\gamma(t,u)), \quad\text{and}\quad
J (\Psi(\tilde{\gamma})(t,u)) = \tilde{J}(\tilde{\gamma}(t,u)).
\end{equation*}
The lemma is an immediate consequence.
\end{proof}

\begin{altproof}{Lemma~\ref{exist-PS}}
By Ekeland's variational principle there exists a Palais-Smale sequence $(s_k,u_k)_k$ for $\tJ|_{\R\times S}$ at the level $c_n$. From $\tJ(s,u)=\tJ(0,s*u)$ we deduce that $(0,s_k*u_k)_k$ is also a Palais-Smale sequence for $\tJ|_{\R\times S}$ at the level $c_n$. Thus we may assume that $s_k=0$. This implies, firstly, that $(u_k)_k$ is a Palais-Smale sequence for $J_S$ at the level $c_n$, and secondly, using $\partial_s\tJ(0,u_k) \to 0$, that \eqref{pohozaev} holds.
\end{altproof}

\begin{lemma}\label{PS}
If the sequence $(u_k)_k$ in $S$ satisfies $J_S'(u_k) \to 0$, $J_S(u_k) \to c > 0$, and \eqref{pohozaev}, then it is bounded and has a convergent subsequence.
\end{lemma}

\begin{proof}
That $(u_k)_k$ is bounded in $E$, hence $u_k\rightharpoonup \bar u$ along a subsequence, can be proved as in \cite[pp.~1644-1644]{Je97}. The compactness of the embedding $H^1_{\text{rad}}(\R^N) \hookrightarrow L^p(\R^N)$ yields $g(u_k) \to g(\bar u)$ in $E^*$. From $J_S'(u_k) \to 0$ it follows that
\begin{equation}\label{eq:u_k}
 -\Delta u_k - \lambda_ku_k - g(u_k) \to 0 \quad \text{in $E^*$}
\end{equation}
for some sequence $\lambda_k \in \R$. Using $J_S(u_k) \to c > 0$ and \eqref{pohozaev}, we deduce as in  \cite[Lemma~2.5]{Je97} that $\lambda_k \to \bar\lambda < 0$ along a subsequence. Then $-\Delta - \bar\lambda$ is invertible and \eqref{eq:u_k} implies $u_k \to (-\Delta - \bar\lambda)^{-1}(g(\bar u))$ in $E$.
\end{proof}

Theorem~\ref{main} follows from \eqref{c_n}, Lemma~\ref{exist-PS} and Lemma~\ref{PS}.

\section*{Acknowledgements}
The second author would like to warmly thank the members of the
Justus-Liebig-University Giessen, where this work was done, for their
invitation and hospitality.

{\sc Address of the authors:}\\[1em]
\parbox{8cm}{Thomas Bartsch\\
 Mathematisches Institut\\
 Universit\"at Giessen\\
 Arndtstr.\ 2\\
 35392 Giessen\\
 Germany\\
 Thomas.Bartsch@math.uni-giessen.de}%\\[1em]
\parbox{7cm}{S\'ebastien de Valeriola\\
 D\'epartement de Math\'ematiques\\
 B\^{a}timent Marc de Hemptinne\\
 Chemin du Cyclotron 2\\
 1348 Louvain-la-Neuve\\
 Belgium\\
 sebastien.devaleriola@uclouvain.be}

\end{document}